\documentclass[12pt,reqno]{article}
\usepackage{enumerate}
\usepackage{fullpage}
\usepackage{amsmath, amssymb, amsthm, url, graphicx,rotating,tkz-graph,relsize}

\newtheoremstyle{plainsl}%
        {\topsep}
        {\topsep}
        {\slshape} 
        {}
        {\normalfont\bfseries}
        {.}
        { }
        {}

\theoremstyle{plainsl}
\newtheorem{theorem}{Theorem}[section]
\newtheorem{proposition}[theorem]{Proposition}
\newtheorem{lemma}[theorem]{Lemma}
\newtheorem{corollary}[theorem]{Corollary}

\newtheorem{remark}[theorem]{Remark}

\newcommand\cref[1]{Corollary~\ref{cor:#1}}

\newcommand\sqr[2]{{\vbox{\hrule height.#2pt
    \hbox{\vrule width.#2pt height#1pt \kern#1pt
        \vrule width.#2pt}\hrule height.#2pt}}}
\renewcommand\qed{%
        \ifmmode\eqno\sqr53
        \else\nolinebreak\ \hfill\sqr53\medbreak\fi}

\numberwithin{equation}{section}

\newcommand{\cC}{{\mathcal C}}

\newcommand{\tI}{\tilde{I}}

\newcommand{\cR}{{\mathcal R}}

\newcommand{\tU}{\tilde{U}}

\newcommand{\tW}{\tilde{W}}

\DeclareMathOperator\spn{span}

\DeclareMathOperator\diam{diam}

\begin{document}
\thispagestyle{empty}
\setcounter{page}{1}
\title{On the connectivity of graphs in association schemes}
\author{Brian G. Kodalen \\
William J. Martin  \\
Department of Mathematical Sciences \\
Worcester Polytechnic Institute \\
Worcester, Massachusetts \\
{\tt \{bgkodalen,martin\}@wpi.edu}}

\date{\today}
\maketitle

\medskip

\begin{abstract}
Let $(X,\cR)$ be a commutative association scheme and let $\Gamma=(X,R\cup R^\top)$ be a connected undirected  graph where $R\in \cR$. Godsil (resp., Brouwer) conjectured that the edge connectivity (resp., vertex connectivity) of 
$\Gamma$ is equal to its valency. In this paper, we prove that the deletion of the neighborhood of any vertex
leaves behind at most one non-singleton component. Two vertices $a,b\in X$ are called ``twins'' in $\Gamma$
if they have identical neighborhoods: $\Gamma(a)=\Gamma(b)$. We characterize twins in polynomial association schemes and show that, in the absence of twins, the deletion of any vertex and its neighbors in $\Gamma$
results in a connected graph. Using this and other tools, we find lower bounds on the connectivity of
$\Gamma$, especially in the case where $\Gamma$ has diameter two.  Among the applications of these results,
we find that the only connected relations in symmetric association schemes which admit a disconnecting set 
of size two are those which are ordinary polygons.
\end{abstract}

\section{Overview}
\label{Sec:overview}

Let $X$ be a finite set of size $v$ and let $\cR=\{R_0,\ldots, R_d\}$ be a partition of $X\times X$ into binary 
relations such that $R_0$ is the identity relation on $X$ and for each $i\in \{1,\ldots,d\}$ 
there exists $i' \in \{1,\ldots,d\}$ such that $R_i^\top = R_{i'}$ where $R^\top = \{(b,a)\mid (a,b)\in R\}$. 
We say $(X,\cR)$ is an \emph{association scheme} (with $d$ \emph{classes})  if there exist integers 
$p_{ij}^k$ ($0\le i,j,k\le d$) such that 
$$ \left| \left\{ c \in X \mid (a,c) \in R_i \ \wedge \ (c,b) \in R_j \right\} \right|  = p_{ij}^k $$
whenever $(a,b)\in R_k$. Throughout this paper, all association schemes are \emph{commutative}: we require
$p_{ij}^k=p_{ji}^k$ for all $i,j,k$.  The problems addressed here immediately reduce to the  
\emph{symmetric} case where $i'=i$ for all $i$; i.e., we will work with symmetric relations only.

Association schemes arise in group theory, graph theory, design theory, coding theory and more. For example, if $X$ is a finite group with conjugacy classes $\cC[g] = \{hgh^{-1}:h\in X\}$ ($g\in X$), then the conjugacy class relations $R_g = \left\{ (a,b) \mid ab^{-1} \in \cC[g]  \right\}$ yield a 
commutative association scheme on the vertex set $X$. The orbits on $X\times X$ of any 
permutation group $G$ acting generously transitively on a set $X$  give a symmetric association scheme.
Some of the most well-studied association schemes are distance-regular graphs, including Moore graphs, distance-transitive graphs, strongly regular graphs, generalized polygons, etc. One studies 
$q$-ary error-correcting codes of length  $n$ as vertex subsets of the Hamming association scheme
$H(n,q)$ \cite[Sec.~9.2]{bcn} and one studies $t$-($v,k,\lambda$) designs as vertex subsets of the 
Johnson association scheme  $J(v,k)$ \cite[Sec.~9.1]{bcn}.  For an introduction to the 
extensive literature on the subject, the reader may consult \cite{del,banito,bcn,godsil}, 
the survey \cite{mtsurvey}, or the more recent book of  Bailey \cite{bailey} which focuses on 
connections to the statistical design of experiments.

Let $(X,\cR)$ be a commutative $d$-class association scheme with \emph{basis relations}
$\cR=\{R_0,\ldots,$ $R_d\}$. For $1\le  i\le d$, we have
a (possibly directed) simple graph $\Gamma_i=(X,R_i)$ on $X$. For $a\in X$, the set $X$ is partitioned into \emph{subconstituents} $R_i(a) = \{ b\in X \mid (a,b)\in R_i \}$ ($0\le i\le d$) with respect to $a$. The association
scheme is \emph{symmetric} if all basis relations are symmetric; each $\Gamma_i$ may be considered as an
undirected graph in this case as $i'=i$ for all $i$. The association scheme  is \emph{primitive} 
\cite[Sec.~2.4]{bcn} if $\Gamma_i$ is connected for all $i=1,\ldots, d$ and \emph{imprimitive} otherwise. A 
\emph{system of imprimitivity} for $(X,\cR)$ is any non-trivial partition of $X$ consisting of the components of
some graph $(X,R)$ where $R$ is a union of basis relations. (The trivial partitions  $\{X\}$ and 
$\{ \{a\} \mid a\in X\}$ are not systems of imprimitivity.)
For each $i$, we  may construct  an undirected graph $H_i$ (possibly with loops) on vertex set 
$\{0,1,\ldots,d\}$, joining  $j$ to $k$ if $p_{ij}^k+p_{ik}^j > 0$. 
We call this the \emph{unweighted distribution diagram} corresponding to basis relation $R_i$.

With reference to a fixed undirected graph $\Gamma$ 
with vertex set $V\Gamma$ and edge 
set $E\Gamma$,  we say that $a$ and $b$ are \emph{twins} if $a\neq b$ yet
$\Gamma(a)=\Gamma(b)$, where $\Gamma(a)$ denotes the set of neighbors of $a$ in graph $\Gamma$.
Write\footnote{Note that some authors assign another meaning to $\bot$; here, we follow \cite[p.\ 440]{bcn}.} 
$a^\bot = \{a \} \cup \Gamma(a)$. A graph $\Gamma$ is \emph{complete multipartite} if any two
non-adjacenct vertices are twins: i.e., the complement of $\Gamma$ is a union of complete graphs.

The main goal of this paper is to prove the following theorem:

\begin{theorem} \label{Tmain}
Let $(X,\cR)$ be a symmetric association scheme. Assume the graph $\Gamma=(X,R_i)$ is connected 
and not complete multipartite. Let
$H=H_i$ be the corresponding unweighted distribution diagram on $\{0,1,\ldots, d\}$. The following are equivalent:
\begin{itemize}
\item[(1)] there exists $a\in X$ for which the subgraph $\Gamma \setminus a^\bot$ is connected;
\item[(2)] for all $a\in X$, the subgraph $\Gamma \setminus a^\bot$ is connected;
\item[(3)]  the subgraph $H \setminus \{0,i\}$ is connected;
\item[(4)] $\Gamma$ contains no twins.
\end{itemize} 
\end{theorem}

We obtain the following corollaries.

\newpage

\begin{corollary}  \label{C2}
Let $(X,\cR)$ be a commutative association scheme. Assume the undirected graph 
$\Gamma=(X,R_i  \cup R_{i'})$ is connected and $a\in X$. Then $\Gamma \setminus 
\Gamma(a)$ contains at most one non-singleton component.
\end{corollary}

\begin{corollary} \label{C1}
Let $(X,\cR)$ be a commutative association scheme.  Assume the undirected graph 
$\Gamma=(X,R_i  \cup R_{i'})$ is connected and $a\in X$.  Then, for any  $T \subseteq a^\bot$ with $\Gamma(a) \not\subseteq T$, the graph  $\Gamma \setminus T$ is connected.
\end{corollary}

\begin{corollary} \label{C3}
Let $(X,\cR)$ be a commutative association scheme. Assume the undirected graph 
$\Gamma=(X,R_i  \cup R_{i'})$ is connected and $C\subseteq  X$ is the vertex set of a clique in $\Gamma$. Then $\Gamma \setminus C$ is connected.
\end{corollary}

The graphs considered in these theorems are all undirected graphs, either a symmetric basis relation in our
association scheme or the symmetrization $(X,R_i  \dot{\cup} R_{i'})$ of some directed basis relation. In both
cases, the edge set of $\Gamma$ is a basis relation of the \emph{symmetrization} $(X,\cR')$ of $(X,\cR)$
where 
$$ \cR' = \left\{ R \cup R^\top \mid R \in \cR \right\} .$$
In this way, the main theorem, while dealing only with the symmetric case, extends immediately to give these
corollaries.

We should remark that these last two results extend naturally to the case where $\Gamma$ is not connected in that the deletion of vertices does not increase the number of components. One verifies this by applying the
respective corollary to the subscheme induced by vertices in a particular component of $\Gamma$.

\section{Connectivity results for highly regular graphs}
\label{Sec:history}

Before we provide proofs of these results and explore various consequences, we now survey earlier
work on the connectivity of graphs in certain association schemes.
 
 Brouwer and Mesner \cite{bromes} showed in 1985 that the vertex connectivity of a strongly regular graph 
 $\Gamma$ is equal to its valency and that the only disconnecting sets of minimum size are the neighborhoods $\Gamma(a)$  of its vertices.  (Brouwer \cite{brouwer} mentions that the corresponding result for edge connectivity
 was established by Ples\'{n}ik in 1975.) This result on vertex connectivity was extended by Brouwer and Koolen 
 \cite{brouwer-koolen} in 2009 to show that a distance-regular graph of valency at least three has vertex connectivity equal to its valency  and that the only disconnecting sets of minimum size are again the neighborhoods $\Gamma(a)$.  Meanwhile a conjecture of Brouwer on the size and nature of the ``second smallest'' disconnecting sets in a strongly regular graph has inspired both new results and interesting examples by Cioab\u{a}, et al.\ \cite{cioaba,cioaba1,cioaba2,cioaba3,cioaba4}.
 
 Godsil \cite{godsil1981} conjectured in 1981 that the edge connectivity of a connected basis relation in any symmetric association scheme is equal to the valency of that graph. Brouwer \cite{brouwer} claimed in 1996 that the same should hold for the vertex connectivity. In \cite{godsil1981}, Godsil proves that if $\Gamma=(X,R_1)$ is 
 regular of valency $v_1$, then the edge connectivity of $\Gamma$ is at least $\frac{v_1}{2} \frac{|X|}{|X|-1}$.
 In 2006, Evdokimov and Ponomarenko proved Brouwer's conjecture for $\Gamma=(X,R_1)$
 in the case when $(X,\cR)$ is equal to the projection onto $X$  of the $v_1$-fold
 tensor product $\bigotimes_{h=1}^{v_1} (X,\cR)$. See \cite{EP1} for definitions and details.
 
 Much more is known about the connectivity of  vertex- and edge-transitive graphs. 
 (See  \cite[Sec.~3.3-4]{godsilroyle}.)  Mader  \cite{mader}  and Watkins \cite{watkins}
 independently obtained the following two results in 1970. The vertex connectivity of an edge-transitive graph
 is equal to the smallest valency. A vertex transitive graph of valency $k$ has vertex connectivity 
 at least $\frac23(k+1)$. Further, in 1971, Mader \cite{mader2} proved that any vertex transitive graph 
 has edge connectivity equal to its valency.

\section{Preliminary results}
\label{Sec:prelim}

In preparation for the proof of our main result, we now prove a few lemmas. We utilize basic terminology and
notation regarding symmetric association schemes. We refer the reader to Section 2.2 of \cite{bcn} for 
basic facts about the Bose-Mesner algebra and Section 2.4 of \cite{bcn} for information on imprimitivity.

Let $A_i$ denote the $01$-matrix with rows and columns indexed by $X$ and $(a,b)$-entry equal to one
if $(a,b)\in R_i$ and equal to zero otherwise. Then the \emph{Bose-Mesner algebra} 
$\spn( A_0,\ldots,A_d)$ is a complex vector space closed under both ordinary and entrywise 
multiplication. So it admits a basis $E_0 = \frac{1}{|X|}J, \ldots,E_d$ of pairwise orthogonal idempotents ($E_i E_j = \delta_{i,j} E_i$)
and the change of basis matrices $[P_{ij}]_{i,j=0}^d$ and $[Q_{ij}]_{i,j=0}^d$ given by 
$$A_j = \sum_{i=0}^d P_{ij}E_i \qquad \text{and} \qquad E_j = \frac{1}{|X|} \sum_{i=0}^d Q_{ij} A_i $$
satisfy $QP= |X| I$ (in particular, $\sum_{j=0}^d Q_{ij} = 0$ for $i\neq 0$) and
$A_i E_j = P_{ji} E_j$ \cite[p.~45]{bcn}, as well as $P_{ji} = \frac{v_i}{m_j} \bar{Q}_{ij}$ where $v_i = P_{0i}$ and $m_j = Q_{0j}$  \cite[Lemma~2.2.1(iv)]{bcn}.

\subsection{Twins}

Let  $\Gamma = (X,R)$ be the graph of a basis relation in $(X,\cR)$.  Write $R(a)=\Gamma(a)$. Examples
where twins arise (i.e., $R(a)=R(b)$ for $a\neq b$) include not only complete multipartite graphs  but   
antipodal distance-regular graphs such as the $n$-cube in which case $R$ is the distance-$\frac{n}{2}$ 
relation of the association scheme.

\begin{lemma} \label{Ltwins}
Let $(X,\cR)$ be a symmetric association scheme and let $\Gamma=(X,R_i)$ for some $i \neq 0$. If $a$ and $b$ 
are twins, then $(X,\cR)$ is imprimitive and some  system of imprimitivity exists in which $a$ and $b$ belong
to the same fibre.
\end{lemma}

\begin{proof}
Denote by $u_j(a)$ the column of $E_j$ indexed by $a\in X$. Then we have, for each $0 \le j\le d$,
$$ P_{ji} u_j(a) = A_i u_j(a) = \sum_{(x, a)\in R_i} u_j(x) =  \sum_{(x,b)\in R_i} u_j(x)   = A_i u_j(b) = P_{ji} u_j(b)$$
so that either $P_{ji}=0$ or $u_j(a)=u_j(b)$. We have $Q_{ij} \neq 0$ if and only if $P_{ji} \neq 0$ from above.
Moreover, for $i\neq 0$, $Q_{i0}+Q_{i1}+\cdots + Q_{id} = 0$. 
Since $Q_{i0}=1$ , we must have $Q_{i\ell } \neq 0$ for some $\ell \neq 0$. Thus $P_{\ell i} \neq 0$ with $\ell \neq 0$ forcing $u_\ell(a) = u_\ell(b)$. Thus $E_\ell$ has repeated columns and the association scheme is imprimitive
\cite[Theorem~2.1]{mmw}. It is well-known that the equivalence classes of the relation $x\cong y \Leftrightarrow u_\ell(x)=u_\ell(y)$ form a non-trivial system of imprimitivity in this case. $\Box$
\end{proof}

\begin{remark} We now discuss twins in polynomial association schemes. These symmetric schemes include
the $P$-{\em polynomial} association schemes where $R_i$ is the distance-$i$ relation in some 
distance-regular graph $\Gamma=(X,R_1)$ and the $Q$-{\em polynomial} or \emph{cometric} association
schemes \cite[Section~2.7]{bcn}.
\begin{enumerate}
\item Assume $(X,\cR)$ is the association scheme coming from a distance-regular graph $\Gamma=(X,R_1)$ with  distance-$k$ relation $R_k$ and assume $R_i(a)=R_i(b)$ for distinct vertices $a$ and $b$. Suppose $a$ and $b$ do not belong to a common antipodal fibre in an antipodal system of imprimitivity. Then $\Gamma$ must be bipartite, in which case columns $a$ and $b$ of $E_j$ can be identical only for $j\in \{0,d\}$ (where  $E_0,\ldots,E_d$ are ordered so
that $P_{01}>P_{11}>\cdots >P_{d1}=-P_{01}$ \cite[Prop.~4.4.7]{bcn}). But then, except for $d=2$, there is 
some $j \neq 0,d$ for which $P_{ji}\neq 0$; thus $a=b$ for $d>2$. So bipartite systems of imprimitivity only 
arise for $d=2$. Viewing complete bipartite graphs as having the antipodal property, we then have that any distinct $a$ and $b$ with  $R_i(a)=R_i(b)$ must belong to the same antipodal fibre, $d$ is even, and $i=d/2$.
\item Assume $(X,\cR)$ is a cometric association scheme, not a polygon, and $a\neq b$ yet 
$R_i(a)=R_i(b)$. Then, by a theorem of Suzuki, et al. \cite{suzimprim,cerzosuz,tanaka2}, $(X,\cR)$ is either 
$Q$-bipartite or $Q$-antipodal. Let $E_0,\ldots,E_d$ be a $Q$-polynomial ordering of the primitive idempotents
and order relations such that $Q_{01}>Q_{11}>\cdots > Q_{d1}$. If $a$ and $b$ belong to the same fibre of a 
$Q$-bipartite imprimitivity system, then $d$ must be even and $i=\frac{d}{2}$ by Corollary 4.2 in \cite{mmw}. Otherwise, $a$ and $b$ must belong to the same $Q$-antipodal fibre and $u_j(a)=u_j(b)$ only for 
$j \in \{0,d \}$.  So $P_{ji}=0$ for $1\le j< d$, forcing $(X,R_i)$ to be an imprimitive strongly regular graph 
(as it is regular with three  eigenvalues). Since the scheme is cometric with an imprimitive strongly regular 
graph as a basis relation, we must have $d=2$ and $a$ and $b$ are non-adjacent vertices in a complete multipartite graph.
\end{enumerate}
\end{remark}

\subsection{The graph homomorphism $\varphi_a$}

For $0< i \le d$, let $\Gamma_i=(X,R_i)$ and let $H_i$ denote the unweighted distribution diagram 
corresponding to symmetric relation $R_i$. 

\begin{proposition}  \label{Pphi}
 For any $a\in X$, the map $\varphi_{a,i} : \Gamma_i \rightarrow H_i$ sending $b\in X$ to $j$ where $(a,b)\in R_j$
 is a graph homomorphism. Under this map, every walk in $\Gamma_i$ projects to a walk in $H_i$ of the same length. As a partial converse, for any $b\in X$ with $(a,b)\in R_{j_0}$ and any walk 
$$ w = (j_0,j_1,\ldots, j_\ell) $$
in $H_i$, there is at least one walk $(b=b_0,b_1,\ldots, b_\ell)$  of length $\ell$ in $\Gamma_i$ such 
that $\varphi_{a,i}(b_s)=j_s$ for each $0\le s \le \ell$.  $\Box$
\end{proposition}

We will call $\varphi_{a,i}$ the \emph{projection map} and will omit the second subscript when 
it is clear from the context.

For vertices $x$ and $y$ in  an undirected graph $\Delta$, we use $d_\Delta(x,y)$ 
to denote the path-length distance from $x$ to $y$ in  $\Delta$, setting $d_\Delta(x,y)=\infty$ when no path from
$x$ to $y$ exists in $\Delta$.

\begin{lemma} \label{LdiamGamma}
Let $(X,\cR)$ be a symmetric association scheme. For $0<i\le d$, let $\Gamma=(X,R_i)$ be a connected
graph and let $H$ denote its unweighted distribution diagram. For $(a,b)\in R_j$, $d_\Gamma(a,b)=d_H(0,j)$.
\end{lemma}

\begin{proof}
A shortest path in $H$ from $j$ to $0$ lifts via $\varphi_a^{-1}$ to a walk in $\Gamma$ from $b$ to
a vertex in $R_0(a)$ --- i.e., lifts to a walk from $b$ to $a$ --- of length $d_H(j,0)$. 
Conversely, each path from $b$ to $a$ in $\Gamma$ projects to a walk of the same length from 
$j$ to $0$ in $H$. $\Box$ 
\end{proof}

\subsection{The decomposition $\{ I_a, U_a, W_a\}$ with respect to a basepoint $a$}
\label{Sec:IUW}

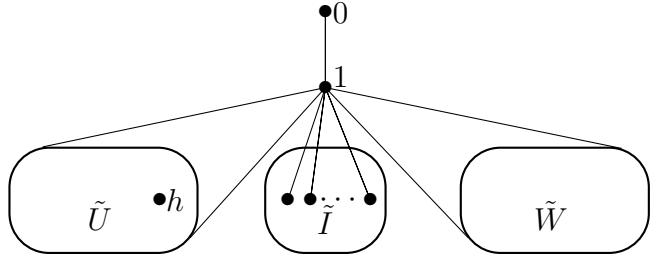
\begin{figure}
\begin{center}
\begin{tikzpicture}
\draw[thick,black]
  (-1.2,-0.7) {[rounded corners=15pt] --
  ++(2.5,0)  -- 
  ++(0,1.4) --
  ++(-2.5,0) --
  cycle};
\draw[thick,black]
  (2.2,-0.7) {[rounded corners=15pt] --
  ++(1.6,0)  -- 
  ++(0,1.4) --
  ++(-1.6,0) --
  cycle};
\draw[thick,black]
  (4.8,-0.7) {[rounded corners=15pt] --
  ++(2.5,0)  -- 
  ++(0,1.4) --
  ++(-2.5,0) --
  cycle};
\node (uu) at (0,-0.2) {$\tilde{U}$};
\node (ii) at (3,-0.25) {$\tilde{I}$};
\node (ww) at (6,-0.2) {$\tilde{W}$};
\node (br0) at (3,2.5) {$\bullet$};
\node (br1) at (3,1.5) {$\bullet$};
\node (brh) at (0.8,0) {$\bullet$};
\node (bri) at (2.2,0) {};
\node (r0) at (3.2,2.5) {$0$};
\node (r1) at (3.2,1.65) {$1$};
\node (rh) at (1,0) {$h$};
\node (bri1) at (2.5,0) {$\bullet$};
\node (bri2) at (2.8,0) {$\bullet$};
\node (brid) at (3.2,0) {$\cdots$};
\node (bri3) at (3.6,0) {$\bullet$};
\draw (-0.76,0.71) -- (3,1.5) -- (3,2.5) -- (3,1.5) -- (6.94,0.69);
\draw (1.2,-0.5) -- (3,1.5);
\draw (4.9,-0.5) -- (3,1.5);
\draw (2.5,0) -- (3,1.5) --  (2.8,0) -- (3,1.5) --  (3.6,0) -- (3,1.5);
\end{tikzpicture} 
\end{center}
\caption{Graph $H$. Upon deletion of $0$ and $1$, the isolated vertices in $\tI$ contain all twins of the basepoint while $\tU$ is vertex set of a component outside $\tI$ which minimizes $\sum_{i\in \tU } v_i$.
\label{FigH}}
\end{figure}

For simplicity, we henceforth take $\Gamma=(X,R_1)$ with unweighted distribution diagram $H=H_1$ in some symmetric association scheme  $(X,\cR)$. We assume throughout the remainder of
Section \ref{Sec:prelim} that $\Gamma$ itself is a connected graph.
We will compare the graphs $\Gamma_a := \Gamma \setminus a^\bot$ and $H' :=H \setminus \{0,1\}$ and show that, with known exceptions, one is connected if and only if the other is connected. One direction is straightforward.

\begin{proposition}
\label{PH'disconn}
If $H'$ is not a connected graph, then for any $a\in X$, $\Gamma_a$ is also disconnected. If 
$i$ and $j$ are in distinct components of $H'$, then   $\Gamma_a$  contains no path from $R_i(a)$ to $R_j(a)$.
\end{proposition}

\begin{proof}
Let $x\in R_i(a)$ and $y\in R_j(a)$ and suppose 
$ x=x_0,x_1,\ldots,x_\ell=y $
is a path in    $\Gamma_a$. Then
$ i = \varphi_a(x_0),  \varphi_a(x_1), \ldots,  \varphi_a(x_\ell) = j$
is a walk from $i$ to $j$ in $H$. Since $H'$ is disconnected, $\varphi_a(x_t)\le 1$ for some $t$ which forces
$x_t \in a^\bot$, a contradiction. $\Box$
\end{proof}

\begin{proposition}
\label{Pxya}
If $x$ and $y$ lie in distinct components of  $\Gamma_a$, then $\Gamma(x)\cap \Gamma(y) \subseteq \Gamma(a)$. $\Box$
\end{proposition}

For $\tU \subseteq \{0,1,\ldots, d\}$, note that $| \varphi_a^{-1}(\tU ) | = \sum_{i\in \tU } v_i$.
We now assume that $H'$ is disconnected and we define a decomposition of its vertex set. Let
$$ \tI = \{ i > 0 \mid  p_{11}^i = p_{11}^0 \} . $$
Now the set $\{2,\ldots,d\} \setminus \tI$ decomposes naturally into the vertex sets of the connected components of $H'$, excluding the isolated vertices in $\tI$. Let $\tU$ be the vertex set of some non-singleton
component such that $| \varphi_a^{-1}(\tU ) |$ is minimized. Let $\tW = \{2,\ldots, d\} \setminus \left( \tI \cup \tU \right)$ as depicted in Figure \ref{FigH}. For 
$x\in X$, set 
$$I_x = \varphi_x^{-1}(\tI ),  \qquad U_x = \varphi_x^{-1}(\tU ),   \qquad W_x = \varphi_x^{-1}(\tW ) $$
and note that $| I_x |$, $|U_x|$, and $|W_x |$ are independent of the choice of $x\in X$. Observe that $x$ and $y$ are twins if and only if $y \in I_x$.  While our basepoint will vary in what follows, our choice of $\tU$, $\tW$ and 
$\tI$ will remain fixed for this connected graph $\Gamma$.

\begin{lemma}  \label{LUdist2}
If $\tW \neq \emptyset$, then for every $u\in U_x$, $d_\Gamma(x,u)=2$.
\end{lemma}

\begin{proof}
By way of contradiction, assume $u \in U_x$ with $\Gamma(x) \cap \Gamma(u) = \emptyset$. For any
$w\in W_x$, we note that $\Gamma$ contains an $xw$-path which does not pass through $u^\bot$.
So $x$ and $w$ lie in the same connected component of $\Gamma_u$. But if $(x,u)\in R_h$ then
$h\in \tU$ so $x\in U_u$ by symmetry. It follows that $W_x \cup \{x\} \subseteq U_u$. But this contradicts
$| \varphi_u^{-1}(\tU ) | \le | \varphi_x^{-1}(\tW ) |$. $\Box$
\end{proof}

\subsection{Comparing the view from multiple basepoints}

\begin{proposition}  \label{PW_b}
For any $a\in X$ and any $b\in U_a$, we have $W_a \cap I_b = \emptyset$.
\end{proposition}

\begin{proof}
If $x$ and $b$ are twins, then $x$ cannot be a twin of $a$ since $b$ is not a twin of $a$. So 
 $\Gamma(x) = \Gamma(b) \subseteq U_a \cup \Gamma(a)$ gives $\Gamma(x) \cap U_a \neq \emptyset$.
 So $x\not\in W_a$.  $\Box$
\end{proof}

Now fix $a\in X$ and choose $b\in U_a$. 
Consider the component $\Delta$ of $\Gamma_b$ containing $a$.
Since $b$ and $a$ are not twins, some element of $\Gamma(a)$ is a vertex of $\Delta$ and hence $\Delta$ contains vertices in $W_a$ unless $\tW = \emptyset$. Let $Z_a = V\Delta \cap W_a$ and let $Y_a= W_a \setminus Z_a$. This vertex decomposition is depicted in Figure \ref{FigGamma}. 
Since $b\in U_a$, we have $a\in U_b$ and, since $\Delta$ is connected, $Z_a \subseteq U_b$.

In the next two results, we proceed under the hypotheses stated at the
beginning of Section \ref{Sec:IUW} and assume that vertices  $a$ and $b \in U_a$ have been chosen and the 
sets $Y_a$ and $Z_a$ are defined as above relative to this pair of vertices.

\begin{lemma} \label{LpathYa}
Let $w=(v_0,v_1,\ldots,v_\ell)$ be a walk in $\Gamma$ with $v_0\in Y_a$ and $v_\ell$ lying some 
other component of $\Gamma_a$. Let $s \in \{1,\ldots,\ell\}$ be the smallest subscript with 
$v_s\not\in Y_a$. Then $v_s \in \Gamma(a)$. $\Box$
\end{lemma}

\begin{lemma} \label{LYZsplit}
For $0\le i \le d$, $R_i(a) \cap Y_a\neq \emptyset$ implies $R_i(a) \cap Z_a \neq \emptyset$. So 
no subconstituent of $\Gamma$ with respect to $a$ is entirely contained in $Y_a$.
\end{lemma}

\begin{proof}
Let $y \in Y_a$ and consider a shortest $ya$-path $\pi$ in $\Gamma$, of length $\ell$ say, and label its 
vertices as follows: $\pi=(y=v_\ell,v_{\ell-1},\ldots,v_1,v_{0} = a)$.
Then, by Lemma \ref{LpathYa}, $v_s \in Y_a$ for  $1 < s \le \ell$.  Consider $j_s=\varphi_a(v_s)$, $0\le s \le \ell$, and assume $j_\ell=i$. Then we have $p_{1,j_{s+1}}^{j_s} >0$ for $0\le s < \ell$.  Note $j_0=0$ and $j_1=1$. Now we lift the walk $(j_0,\ldots,j_\ell)$ in $H$ to a different walk in $\Gamma$. Since $a$ and $b$ are not twins,
we may choose $v'_1 \in \Gamma(a) \setminus \Gamma(b)$.
Since $p_{1j_2}^{1}>0$, there exists $v'_2\in R_{j_2}(a)$ with $v'_2$ adjacenct to  $v'_1$ in $\Gamma$. Continuing in this manner, we
may construct a walk $\pi'=(a=v'_0,v'_1,\ldots, v'_\ell)$ in $\Gamma$ with $\varphi_a(v'_s)=j_s$. 
Since $\Gamma(b) \subseteq \Gamma(a) \cup U_a$, none of the vertices $v'_s$ lie in 
$\Gamma(b)$, so the entire walk $\pi'$ is contained in one component of $\Gamma_b$. By definition of 
$Z_a$, we then have $v'_\ell \in Z_a \cap R_i(a)$. $\Box$

\end{proof}

\begin{figure}
\begin{center}
\begin{tikzpicture}
\draw[thick,black]
  (0,0) {[rounded corners=15pt] --  ++(2.25,0)  --   ++(0,5) --  ++(-2.25,0) --  cycle};
  \draw[thick,black]
  (2.5,0) {[rounded corners=15pt] --  ++(1.5,0)  --   ++(0,5) --  ++(-1.5,0) --  cycle};
  \draw[thick,black]
  (4.25,0) {[rounded corners=15pt] --  ++(2.25,0)  --   ++(0,5) --  ++(-2.25,0) --  cycle};
  \draw[thick,black]
  (0,5.25) {[rounded corners=15pt] --  ++(6.5,0)  --   ++(0,1) --  ++(-6.5,0) --  cycle};  
 \draw[thick,black]  (4.25,4) -- (6.5,2);
 \draw[thick,black]  (0.9,0) -- (0.9,5);
 \draw[thick,black]  (0.9,5.25) -- (0.9,6.25);
\node (bra) at (3.25,6.6) {$\bullet$};  \node (ra) at (3.5,6.6) {$a$};
\node (brb) at (0.25,2.5) {$\bullet$}; \node (rb) at (0.5,2.5) {$b$};
\node (li) at (3.2,5.75) {$\Gamma(a)$};
\node (lu) at (1.35,0.75) {$U_a$};
\node (li) at (3.2,0.75) {$I_a$};
\node (lw) at (5.45,0.75) {$W_a$};
\node (ly) at (5,2.75) {$Y_a$};
\node (lz) at (5.5,3.5) {$Z_a$};
\node (lbb1) at (0.55,3.5) {$b^\bot$};
\node (lbb2) at (0.55,5.75) {$b^\bot$};
\fill[gray,opacity=.5]   
 (0.9,5) {[rounded corners=15pt] --   (2.25,5)  -- (2.25,3.5) -- (0.9,3.5)} -- cycle; 
  \fill[gray,opacity=.5]  
  (4.25,4) {[rounded corners=15pt] --  (4.25,5)  -- (6.5,5)}  -- (6.5,2) --  cycle;
\fill[gray,opacity=.5]   
 (0.9,5.25) {[rounded corners=15pt] --   (6.5,5.25)  --  (6.5,6.25)} -- (0.9,6.25) -- cycle;
\fill[gray,opacity=.5]   
 (2.5,0) {[rounded corners=15pt] --   ++(1.5,0)  --   ++(0,5) --  ++(-1.5,0) --  cycle};
\end{tikzpicture} 
\end{center}
\caption{This diagram depicts $\Gamma$ as decomposed relative to basepoint $a$. In $\Gamma_b$,
vertex $a$ belongs to component $\Delta$, whose vertex set is indicated by the shaded region. 
The set $W_a$ splits into $Z_a$ and $Y_a$ according to membership in $V\Delta$.
\label{FigGamma}}
\end{figure}
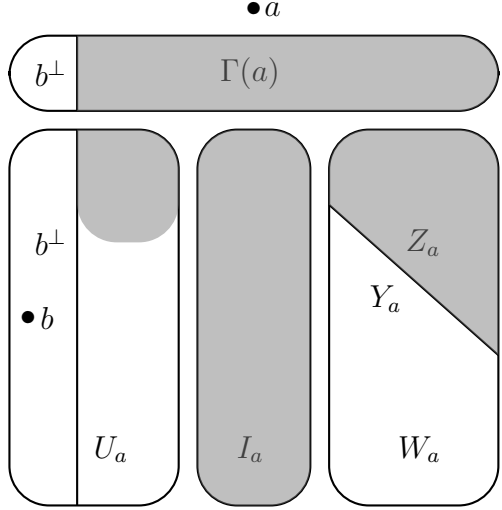

\begin{lemma} \label{LWa-empty}
If $\tW \neq \emptyset$, then $\Gamma$ has diameter two; i.e., $p_{11}^i > 0$ for all $i>1$.
\end{lemma}

\begin{proof}
Let $a,x \in X$ with $x \not\in a^\bot$. Choose $b\in U_a$ as above and consider, in turn, each part of
the decomposition 
$$ V\Gamma_a \ = \ I_a \ \dot{\cup} \ U_a  \ \dot{\cup} \  Z_a  \ \dot{\cup}  \  Y_a $$
relative to $a$ and $b$.
If $x\in I_a$, $\Gamma(x) = \Gamma(a)$; if $x \in U_a$, then $d_\Gamma(a,x)=2$ by Lemma \ref{LUdist2}.
Next consider $x\in Z_a$. Then $d(x,b)=2$ but $\Gamma(x)\cap \Gamma(b) \subseteq \Gamma(a)$ since 
$x$ and $b$ lie in distinct components of $\Gamma_a$. Finally, consider $x\in Y_a$ with $(a,x)\in R_i$.
By Lemma \ref{LYZsplit}, there exists $x'\in Z_a \cap R_i(a)$. Since $x'$ has a neighbor in $\Gamma(a)$,
$p_{11}^i > 0$ which then implies that some neighbor of $x$ lies in $\Gamma(a)$ as well. $\Box$
\end{proof}

\begin{theorem}  \label{TWempty}
Let $(X,\cR)$ be any symmetric association scheme and let $\Gamma=(X,R_1)$ be
any connected basis relation. With reference to the above definitions, $\tW = \emptyset$.
\end{theorem}

\begin{proof}
By way of contradiction, assume $\tW \neq \emptyset$ and define
$$ \mu = \min \{ p_{11}^i \mid i \in \tU \},  \qquad \omega = \min \{ p_{11}^i \mid i \in \tW \}$$
and select $k\in \tU$ and $\ell \in \tW$ with $p_{11}^k = \mu$ and $p_{11}^\ell = \omega$. Note that
$\mu > 0$ and $\omega > 0$ by Lemma \ref{LWa-empty}.
Now choose $a\in X$, and select $x$ in $R_k(a)$.  Since $x$ is not a twin of $a$, we may choose
$a'\in \Gamma(a) \setminus \Gamma(x)$ and since $p_{1\ell}^1>0$, we may choose and $y$ in $R_\ell(a)$ which is a neighbor of 
$a'$.  Since $\Gamma_x$ contains a path from $a$ to $y$ and $a\in U_x$, we have $y\in U_x$. So
$| \Gamma(x) \cap \Gamma(y) | \ge \mu$.  By Proposition \ref{Pxya}, $\Gamma(x)\cap \Gamma(y) \subseteq \Gamma(a)$. (See Figure \ref{FigWempty}.) But $a' \in \Gamma(y)\cap \Gamma(a)$. So 
$$\omega \ge 1 + | \Gamma(x) \cap \Gamma(y) | > \mu.$$

Now we simply reverse the roles of $x$ and $y$; more precisely, we swap $\ell$ and $k$.

Select $x$ in $R_\ell(a)$ and,  choosing
$a'\in \Gamma(a) \setminus \Gamma(x)$, we may find a vertex $y$ in $R_k(a)$ which is a neighbor of 
$a'$.  Since $\Gamma_x$ contains a path from $a$ to $y$ and $a\in W_x$, we have $y\in W_x$. So
$| \Gamma(x) \cap \Gamma(y) | \ge \omega$.  By Proposition \ref{Pxya}, $\Gamma(x)\cap \Gamma(y) \subseteq \Gamma(a)$. But $a' \in \Gamma(y)\cap \Gamma(a)$. So 
$$\mu \ge 1 + | \Gamma(x) \cap \Gamma(y) | > \omega.$$
We have $\omega > \mu$ and $\mu > \omega$, producing the desired contradiction. $\Box$
\end{proof}

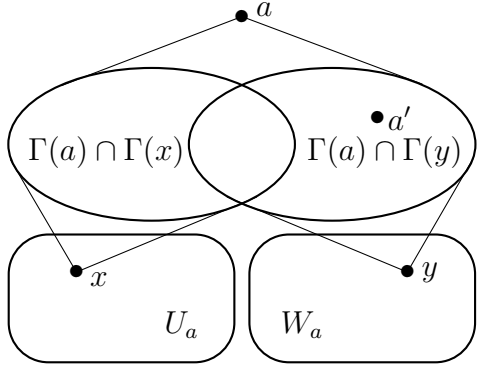
\begin{figure}
\begin{center}
\begin{tikzpicture}
\draw[thick,black] (1.8,0) ellipse (0.75in and 0.4in);
\draw[thick,black] (4.2,0) ellipse (0.75in and 0.4in);
\draw[thick,black]
  (-.1,-2.9) {[rounded corners=15pt] --
  ++(3.,0)  -- 
  ++(0,1.7) --
  ++(-3.,0) --
  cycle};
\draw[thick,black]
  (3.1,-2.9) {[rounded corners=15pt] --
  ++(3,0)  -- 
  ++(0,1.7) --
  ++(-3,0) --
  cycle};
\node (uu) at (2.2,-2.4) {$U_a$};
\node (ww) at (3.8,-2.4) {$W_a$};
\node (ba) at (3,1.7) {$\bullet$};   \node (la) at (3.3,1.8) {$a$};
\node (bx) at (.8,-1.7) {$\bullet$};   \node (lx) at (1.1,-1.8) {$x$};
\node (by) at (5.2,-1.7) {$\bullet$};   \node (ly) at (5.5,-1.7) {$y$};
\node (bap) at (4.8,.35) {$\bullet$};   \node (lap) at (5.1,.35) {$a'$};
 \node (lgamax) at (1.2,-.1) {$\Gamma(a)\cap \Gamma(x)$};
 \node (lgamay) at (4.9,-.1) {$\Gamma(a)\cap \Gamma(y)$};
\draw (-0.08,-0.2) -- (.8,-1.7) -- (3.2,-0.7);
\draw (6.08,-0.18) -- (5.2,-1.7) -- (2.8,-0.7);
\draw (0.6,0.78) -- (3,1.7) -- (5.4,0.8);
\end{tikzpicture} 
\end{center}
\caption{Since $\Gamma$ has diameter two, all common neighbors of $x$ and $y$ are 
contained in  $\Gamma(a)$. \label{FigWempty}}
\end{figure}

\section{Proofs of the main theorem and its corollaries}
\label{Sec:proofs}

We are now ready to present the proof of our main theorem.

\bigskip

\noindent {\sl Proof of Theorem \ref{Tmain}.} \ For notational convenience, we assume $i=1$.

We begin by showing $(3) \Leftrightarrow (4)$. If $a$ and $b$ are twins in $\Gamma$ with $(a,b)\in R_j$, then
$j > 1$ and $p_{11}^j = v_1$ so that $j\in \tI$ and $\{ j\}$ is the entire vertex set of some component of 
$H' = H   \setminus \{0,1\}$. So either $H'$ is not connected or $d=j=2$ and $\Gamma$, being imprimitive,
is a complete multipartite graph.  Conversely, by Theorem \ref{TWempty}, $\tW=\emptyset$ so if $\tI=\emptyset$
we have that  $H'$ is connected.

The assertion $(2) \Rightarrow (1)$ is trivial. Proposition \ref{PH'disconn} gives us $(1) \Rightarrow (3)$. So
we need only check that $(3)$ implies $(2)$.

Assume now that $H'$ is connected and yet there is some $a\in X$ with $\Gamma_a$ not connected. 
By Proposition \ref{Pphi}, any $x$ in $\Gamma_a$ is joined by a walk in $\Gamma_a$ to some vertex in 
$R_j(a)$ for every $j>1$. (Simply lift a walk in $H'$ from $\ell$ to $j$ where $(a,x)\in R_\ell$.)
So for every $j>1$ every connected component of $\Gamma_a$ intersects every 
subconstituent $R_j(a)$  non-trivially.  
Select $j>1$ so as to maximize $D:=d_H(0,j)$ and choose $x,y \in R_j(a)$ such that $x$ and $y$ lie in
distinct components of $\Gamma_a$.  Then every $xy$-path in $\Gamma$ must include a vertex in $\Gamma(a)$,
so $d_\Gamma(x,y) \ge 2(D-1)$. Since $d_\Gamma(x,y) \le D$ by Lemma~\ref{LdiamGamma}, this forces $D\le 2$. In particular, $p_{11}^\ell >0$ for every $\ell > 1$.

Select $\ell >1$
so as to minimize $p_{11}^\ell$ and select $x,y \in R_\ell(a)$ from distinct components of $\Gamma_a$.
Then $(x,y) \in R_j$ for some $j>1$ and so $|\Gamma(x)\cap \Gamma(y)| \ge p_{11}^\ell$. But 
since these two vertices lie in distinct components, Proposition \ref{Pxya} gives us
$$ \Gamma(x) \cap  \Gamma(y) \subseteq \Gamma(a) \cap \Gamma(y)$$
so $p_{11}^j = p_{11}^\ell$ and $ \Gamma(x) \cap  \Gamma(y) =\Gamma(a) \cap \Gamma(y)$.
If $a' \in \Gamma(a)$, then $a'$ has $p_{1\ell}^1 > 0$ neighbors in $R_\ell(a)$. For any such neighbor $z$, we
must have either $\Gamma(z) \cap \Gamma(a) = \Gamma(x) \cap \Gamma(a)$ or $\Gamma(z) \cap \Gamma(a) = \Gamma(y) \cap \Gamma(a)$, both of which force $a' \in
\Gamma(x)$. So vertices $a$ and $x$ must be twins.
$\Box$

The proofs of Corollaries \ref{C1}, \ref{C2} and \ref{C3} are now rather immediate. Since each is a statement about the  symmetrization of some commutative scheme, Theorem \ref{Tmain} applies.

\bigskip

\noindent {\sl Proof of Corollary \ref{C2}.}  This is essentially Theorem \ref{TWempty}. $\Box$

\bigskip

\noindent {\sl Proof of Corollary \ref{C1}.}  First, if we have no twins then
$\Gamma_a$ is connected. Any 
$a' \in \Gamma(a)$ has at least one neighbor in $V\Gamma_a$. If $a\not\in T$, then some $a' \in \Gamma(a)$
is also not included in $T$. So the graph $\Gamma\setminus T$ is 
connected as long as $T \neq \Gamma(a)$. 

If $b$ is a twin of $a$ in $\Gamma$, then $\{b\}$ is adjacent to every $x\in \Gamma(a)$. 
Since $\Gamma(a) \not\subseteq  T$,    some $a' \in \Gamma(a)$ is a vertex of $\Gamma \setminus T$.
By Corollary \ref{C2}, $\Gamma \setminus a^\bot$ has at most one non-singleton component. Let 
$\Xi$ be the component of $\Gamma \setminus T$ containing this component as a connected subgraph.
(If $\Gamma \setminus a^\bot$ consists only of singletons, choose $\Xi$ to be any component
of $\Gamma \setminus a^\bot$.) Since $a'$ has at least one neighbor in $V\Gamma_a$, 
the component $\Xi$ contains $a'$ and every twin $b$ of $a$ since each of these is a neighbor of $a'$.
Likewise, if $a\not\in T$, then $a$ belongs to $\Xi$ since it is adjacent to $a'$. So in this case as well,
$\Gamma \setminus T$ is connected.
 $\Box$

\bigskip

\noindent {\sl Proof of Corollary \ref{C3}.} Let $a\in C$ and take $T=C$. Then apply Corollary \ref{C1}. $\Box$

We finish this section with a simple generalization arising from the proof above. We state the result without proof.

\begin{theorem}
Assume $(X,\cR)$, $\Gamma$ and $H$ are defined as in Theorem \ref{Tmain}.
Let $B_{H,t}(0) = \{ i \mid 0\le i\le d, \ d_H(0,i)\le t\}$ and $B_{\Gamma,t}(a) = \cup_{B_{H,t}(0)} R_i(a)$.
\begin{itemize}
\item[(a)] 
If $\Gamma' := \Gamma \setminus B_{\Gamma,t}(a)$ is disconnected and $b\in X$ with $d_\Gamma(a,b)=D$ (the diameter of $\Gamma$), then for any $x\not\in B_{\Gamma,t}(a)$ not in the same component of $\Gamma'$ as 
$b$, we have $d_\Gamma(a,x) \le 2t$. 
\item[(b)]
If $H \setminus B_{H,t}(0)$ is connected and yet $\Gamma \setminus B_{\Gamma,t}(a)$ is disconnected,
then $D \le 2t$.  $\Box$
\end{itemize}
\end{theorem}

\section{Further results on connectivity}
\label{Sec:conn}

In this section, we develop some machinery for the study of small disconnecting sets which are not
localized. We then apply these tools to show that, with the exception of polygons, a basis relation in a  symmetric association scheme has vertex connectivity at least three. We can say a bit more in the case where $\Gamma$
has diameter two.

Elementary graph theoretic techniques allow us to handle the case where $\Gamma$ is in some sense locally connected. For example, if $\Gamma(y)$ induces a connected subgraph for every $y\in T$ and $d_\Gamma(y,y')\ge 3$ for any pair of distinct elements $y,y'\in T$, then $\Gamma \setminus T$ is connected. 
The proof of this claim is essentially the same as the proof of the following proposition.

\begin{proposition}
\label{Psptree}
Let $\Gamma$ be a connected simple graph. Suppose any two vertices at distance two in  $\Gamma$ lie in some common cycle of  length at most  $g$ and $T \subseteq V\Gamma$ satisfies $d_\Gamma(y,y') 
\ge g+1$ for all pairs $y,y'$ of distinct vertices from $T$. Then $\Gamma \setminus T$ is connected.
\end{proposition}

\begin{proof}
Set $\delta = \lfloor g/2 \rfloor$ and, for $y\in T$ set 
$B_\delta(y) =  \{ x\in X \mid d_\Gamma(x,y) \le \delta \}$. The 
induced subgraph $\Gamma[B]$ of $\Gamma$ determined by $B=B_\delta(y)$ is connected so 
admits a spanning tree. Moreover, since $y$ is not a cut vertex of $\Gamma[B]$,
there exists a spanning tree $T_y$ for $\Gamma[B]$ in which $y$ is a leaf vertex. For $y\in T$, 
let $E_y$ denote the edge set of $T_y$ with the sole edge incident to $y$ removed.

Now consider the minor $\Delta$
of $\Gamma$ obtained by contracting $B_\delta(y)$ to a single vertex for every $y\in T$. Since $\Delta$ is
again a connected graph, it admits a spanning tree $T$. Lift the edge set $E_T$ of $T$ back to $E\Gamma$
and note that $E_T$ contains no edge from any of the induced subgraphs $\Gamma[B_\delta(y)]$, $y\in T$.
So $E_T \cup \left( \cup_{y\in T} E_y \right)$ is the edge set of a spanning tree in $\Gamma \setminus T$,
which demonstrates that   $\Gamma \setminus T$ is connected. $\Box$
\end{proof}

\subsection{A spectral lemma}
\label{Sec:spec}

Eigenvalue techniques such as applications of eigenvalue interlacing play an important role in \cite{bromes} and \cite{brouwer-koolen}. The following lemma is inspired by those ideas. This can be used, in conjunction with
Lemma \ref{PK211}, to show that a graph with a small disconnecting set $T$ whose elements are not too close
together must be locally a disjoint union of cliques of size at most $|T|$.

\begin{lemma}
\label{Lspec-cut}
Let $(X,\cR)$ be a symmetric association scheme and let $\Gamma=(X,R_1)$ be the graph associated to a con\-nec\-ted basis relation.
Assume that $\Gamma$ contains no induced subgraph isomorphic to $K_{2,1,1}$. 
If  $T\subseteq X$ is a disconnecting set for $\Gamma$, then $|T|> p_{11}^1$.
\end{lemma}

\begin{proof}
The result obviously holds when $\Gamma$ is complete multipartite, so assume $\Gamma$ is not a com\-plete 
multipartite graph.  By \cite[Cor.~3.5.4(ii)]{bcn}, we then know that the second largest eigenvalue $\theta$ of $\Gamma$ is positive. Order the eigenspaces of the scheme so that $A_1 E_1 = \theta E_1$ and abbreviate 
$E=E_1$. For $K,L\subseteq X$, denote by $E_{K,L}$ the submatrix of $E$ obtained by restricting
to rows in $K$ and columns in $L$. Let $C$ be any
clique in $\Gamma$. Then, because $v_1 > \theta > 0$, the matrix $E_{C,C} =  \frac{m_1}{|X|}I
+ \frac{\theta m_1}{v_1|X|}(J-I)$ is invertible. 

Assume now that some disconnecting set $T \subseteq X$ has $|T| \le p_{11}^1$.
Let $\Xi$ and $\Xi'$ be two connected components of $\Gamma \setminus T$ with 
vertex sets $B$ and $ B'$, respectively,  and let
$\rho$ and $\rho'$ denote the spectral radii of these two graphs.  Assume, without loss, that $\rho\le \rho'$. By eigenvalue interlacing, $\rho \le \theta$.  (see, e.g.,\cite[Theorem~3.3.1]{bcn}.) We now show $\rho=\theta$.

Since $\Gamma$ does not contain $K_{2,1,1}$ as an induced subgraph, it is locally a disjoint union of cliques
and every edge of $\Gamma$ lies in a clique $C$ of size $p_{11}^1 + 2$.  If $\Xi$ is edgeless, then $T$
contains all neighbors of some vertex, which is impossible since $|T| \le p_{11}^1 < v_1$. So $\Xi$
contains at least one edge and  $B\cup T$ contains some clique $C$ 
of size at least $p_{11}^1 + 2$. It follows that the submatrix  $E_{X,B\cup T}$
has rank at least $p_{11}^1 + 2$.  But $|T|\le p_{11}^1$. So the row space of $E_{X,B\cup T}$ 
contains at least two linearly  independent vectors which are zero in every entry indexed by an 
element of $T$. Restricting these two vectors to coordinates in $B$ only, we obtain two linearly
independent eigenvectors for graph $\Xi$ belonging to eigenvalue $\theta$. It follows that
$\rho = \theta$ and $\rho$, the spectral radius of $\Xi$, is not a simple eigenvalue. This contradicts
the Perron-Frobenius Theorem  (see, e.g., \cite[Theorem~3.1.1]{bcn}) since $\Xi$ was chosen to
be a connected graph.  $\Box$ 
\end{proof}

\begin{remark}
The hypotheses of the above lemma may clearly be weakened. The proof simply requires that both $B\cup T$ and 
$B' \cup T$ contain cliques of size $|T|+2$ or larger and that the entries $E_{xy}$ of idempotent $E$ 
are the same for all adjacent $x$ and $y$ in $V\Gamma$.
\end{remark}

\subsection{Intervals and metric properties of $\Gamma$}
\label{Sec:intervals}

For $a,b\in X$, if  $(a,b)\in R_i$, Lemma \ref{LdiamGamma} tells us that 
the path-length distance  $d_\Gamma(a,b)$ between $a$ and $b$ in 
graph $\Gamma$  is equal to the path-length distance $d_H(0,i)$ between $0$ and $i$ in $H$. It follows
that the diameter, $D$ say, of $\Gamma$ is equal to $\max_i d_H(0,i)$, which happens to be the diameter 
of $H$. We thus partition the index set  $\{0,1,\ldots,d\}$ according to distance from $0$ in $H$.
For each $0\le h\le D$, define $I_h = \{ i : d_H(0,i) = h\}$. 
For $0 \le i \le d$ with $i \in I_h$, define   
$$c(i) = \sum_{ j \in I_{h-1} } p_{1j}^i ~.$$

\begin{proposition}
\label{Lc(i)}
With $c(i)$ defined as above
\begin{itemize}
\item[(a)]
For any geodesic $0=\ell_0, 1=\ell_1,\ell_2,\ldots, \ell_h$ in $H$, 
$$1= c(\ell_1) \le c(\ell_2) \le \cdots \le c(\ell_h). $$
\item[(b)]
If $c(i)=1$, then for any $\ell \in \{1,\ldots, d\}$ which lies along a geodesic from $0$ to $i$ in $H$, $c(\ell)=1$ as well.
\item[(c)]
If $c(i)=1$, then there is a unique shortest path in $H$ from $0$ to $i$ and, for $(a,b)\in R_i$, there is a unique
shortest path in $\Gamma$ from $a$ to $b$.
\end{itemize}
\end{proposition}

\begin{proof}
For part (a), observe that for $(a,b)\in R_{\ell_h}$ there exists $a' \in R_{\ell_{h-1}}(b)$ adjacent  to $a$ since
$p_{1,\ell_{h-1}}^{\ell_h} > 0$ so that 
$$ \{ x\in X \mid (x,b) \in R_1, \ d_\Gamma(x,a') = d_\Gamma(b,a')-1 \} \! \subseteq  \! 
\{ x\in X \mid (x,b) \in R_1, \ d_\Gamma(x,a) = d_\Gamma(b,a)-1 \}. $$ 
Parts (b) and (c) follow immediately. $\Box$
\end{proof}

\medskip

For $a,b\in X$, we define the \emph{interval} $[a,b]$ to be the union of the vertex sets of all geodesics in 
$\Gamma$ from $a$ to $b$:
$$ [a,b] = \left\{ x\in X \mid d_\Gamma(a,x) + d_\Gamma(x,b) = d_\Gamma(a,b) \right\}. $$

For the purpose of the present discussion, we introduce a piece of terminology. For $x\in X$ and 
$y\in T\subseteq X$,  we say that $x$ is \emph{proximal} to $y$ (relative to $T$) if 
$d_\Gamma(x,y) \le d_\Gamma(x,y')$ for all $y'\in T$. Vertex $x$ is then \emph{proximal only} to $y\in T$ if 
$d_\Gamma(x,y) < d_\Gamma(x,y')$ for all $y'\in T$ distinct from $y$.

\begin{proposition}
\label{Pci=1}
Let $T$ be a disconnecting set for $\Gamma$  and let $x$ and $z$ be vertices lying in different components of  $\Gamma \setminus T$  with $(x,z)\in R_i$. Suppose there is some $y\in T$ with $z \sim y$.
If either $x$ is proximal only to $y$ or $x$ is proximal to $y$ and $z$ is proximal only to $y$,  then $c(i)=1$.
\end{proposition}

\begin{proof}
Apply the triangle inequality. $\Box$ 
\end{proof}

\subsection{Small disconnecting sets}
\label{Sec:smallcut}

We begin by examining a simple condition which guarantees that $\Gamma$ is locally a disjoint union of 
cliques.

\begin{lemma}
\label{PK211}
Let $T$ be a disconnecting set for $\Gamma$, $y\in T$. Suppose $d_\Gamma(y,y')\ge 3$ for all $y'\in T$ with $y'\neq y$. Then $\Gamma$ is $K_{2,1,1}$-free.
\end{lemma}

\begin{proof}
Let $j\in I_2$ and let $z \in R_j(y)$. Let $x \sim y$ be some vertex lying in a different component of 
 $\Gamma \setminus T$ from that containing $z$. For $(x,z)\in R_i$, we find $c(i)=1$. So $c(j)=1$ by Proposition
 \ref{Pci=1}.  $\Box$ 
\end{proof}

\begin{lemma}
\label{Pfarawayx}
Let $T$ be a disconnecting set for $\Gamma$, $y\in T$.
\begin{itemize}
\item[(a)] Let $x$ and $z$ be vertices lying in different components of $\Gamma \setminus T$. If
$d_\Gamma(x,y')+d_\Gamma(y',z) > D$ for every $y'\in T$ except $y$, then $z$ has a 
unique neighbor lying closer to $x$ and $z$ has a unique neighbor lying closer to $y$.
\item[(b)] Suppose $x\in X$ satisfies $d_\Gamma(x,y')=D$ for every $y'\in T$ except $y$. If $z\in X$ lies in a component of $\Gamma \setminus T$ distinct from that containing $x$, then $z$ has a unique neighbor 
lying closer to $x$ and $z$ has a unique neighbor lying closer to $y$. 
\end{itemize}
In both cases, for $(x,z)\in R_i$, and $(y,z)\in R_j$,  we have $c(i)=c(j)=1$.
\end{lemma}

\begin{proof}
Clearly (b) follows from (a). So first verify (a) for the case $z\sim y$. Next, 
observe that any geodesic joining $x$ to $z$ passes through $y$. So $ [x,z] = [x,y] \cup [y,z]$.
Let $z' \in \Gamma(y) \cap [y,z]$. Since $[x,y] 
\subseteq [x,z]$ and $[x,z'] = [x,y] \cup \{z'\}$, 
we find $\Gamma(x) \cap [x,z] = \Gamma(x) \cap [x,z']$, a set of size one. 
By the same token, $[y,z] \subseteq [x,z]$ and so $\Gamma(z) \cap [y,z] \subseteq 
\Gamma(z) \cap [x,z]$ gives $|\Gamma(z) \cap [y,z] |=1$.  $\Box$
\end{proof}

\begin{lemma}
\label{Lxp111}
Let $T$ be a disconnecting set for $\Gamma$, $y\in T$,  and suppose $x\in X$ satisfies $d_\Gamma(x,y')=D$ for every $y'\in T$ except $y$. Then
\begin{itemize}
\item[(a)]
 for $(x,y)\in R_i$ where $i\in I_h$, we have
$ \sum_{\ell \in I_h} p_{1\ell}^i = p_{11}^1$.
\item[(b)]
for $z \in X\setminus T$ which is separated from $x$ by deletion of $T$, if 
$\Gamma(z) \cap T \subseteq \{y\}$, then
$ \sum_{\ell \in I_k} p_{1\ell}^j = p_{11}^1$ where $(y,z)\in R_j$ with $j\in I_k$.
\end{itemize}
\end{lemma}

\begin{proof}
Let $z$ be a neighbor of $y$ which is separated from $x$ by deletion of $T$.
Since $d_\Gamma(x,z) \le D$, we see that $x$ is proximal only to $y$ and $[x,z]=[x,y] \cup \{z\}$.  
The set  $\Gamma(y)\cap \Gamma(z)$ has size $p_{11}^1$ and every $z' \in \Gamma(y)\cap \Gamma(z)$ 
lies at distance $h+1$  from $x$ in $\Gamma$. Since every other neighbor of $z$, with the exception 
of $y$, is further away from  $x$, we have $\sum_{ \ell \in I_{h+1} } p_{1 \ell}^j = p_{11}^1$ where 
$(x,z)\in R_j$. Reversing roles, we
see that $x$ then has exactly $p_{11}^1$ neighbors which lie at distance $h+1$ from $z$. But, for
$x' \sim x$, $d_\Gamma(x',y)=d_\Gamma(x',z)-1$. This gives (a). To obtain (b), observe that 
every neighbor $x'$ of $x$ with $d_\Gamma(x',z)=d_\Gamma(x,z)$ must have  $d_\Gamma(x',y)=
d_\Gamma(x,y)$. By part (a), there are exactly $p_{11}^1$ such vertices. So, for $(x,z)\in R_s$, 
$ \sum_{\ell \in I_{h+k}} p_{1\ell}^s = p_{11}^1$. Reversing roles, we see that exactly $p_{11}^1$
neighbors of $z$ lie at distance $h+k$ from $x$. But this is precisely the set of vertices adjacent 
to $z$ which lie at distance $k$ from $y$.  $\Box$
\end{proof}

\begin{theorem}
\label{Tcutsize2}
Let $(X,\cR)$ be a symmetric association scheme and let $\Gamma=(X,R_1)$  be the graph 
associated to a connected basis relation. If $\Gamma$ admits a  disconnecting set of size two, 
then $\Gamma$ is isomorphic to a polygon.
\end{theorem}

\begin{proof}
Let $T=\{y,y'\}$ be a disconnecting set of size two. Let  $D=\diam \Gamma$ and let $B$  be the vertex set of some connected component  of $\Gamma \setminus T$. First consider the 
case where $y'$ is the unique vertex at distance $D$ from $y$ in $\Gamma$. Then every vertex
is at distance $D$ from exactly one other vertex. On the other hand,  if $x\in B \cap \Gamma(y)$, then 
any neighbor of $y'$ not lying in $B$ must be at distance $D$ from $x$ by the triangle inequality. It follows that 
$y$ has exactly one neighbor not in $B$ and, symmetrically, exactly one neighbor in $B$. So the
graph has valency two in this special case.

By Corollary \ref{C1}, we have
$d_\Gamma(y,y')\ge 3$ so that $\Gamma$ is $K_{2,1,1}$-free by Lemma \ref{PK211}.   Let $x$ (resp., $x'$) 
denote some vertex  at distance $D$ from $y'$ (resp., $y$), with $x\neq y$, $x' \neq y'$. Let $B$ and $B'$ be the vertex sets of  two connected components  $\Xi$ and $\Xi'$, respectively, of $\Gamma \setminus T$ and assume
$x\in B$. By Lemma  \ref{Pfarawayx}(b), any $z\in B'$ has a unique neighbor lying closer to $y$.  By 
Lemma \ref{Lxp111}(a), any $z \in B' \setminus \Gamma(y')$ has exactly $p_{11}^1$ neighbors $z'$ satisfying
$d_\Gamma(z',y)=d_\Gamma(z,y)$.  Since $d_\Gamma(x,y)+d_\Gamma(y,x') > D$ and $d_\Gamma(x,y')
+d_\Gamma(y',x') > D$, we must have $x' \in B$ also. So we can swap the roles of $x$ and $x'$, $y'$ and $y$,
to find that any $z  \in B' \setminus \Gamma(y)$  has a unique neighbor closer to $y'$ and exactly $p_{11}^1$
neighbors $z'$ with  $d_\Gamma(z',y')=d_\Gamma(z,y')$. Now select $z\in B'$ so as to maximize 
$d_\Gamma(z,y) + d_\Gamma(z,y')$.  By Corollary \ref{C1}, $d_\Gamma(y,y') \ge 3$, so we may assume
$z$ is not adjacent to $y'$. Then $z$ has a unique neighbor lying closer to $y$ and exactly $p_{11}^1$ neighbors
$z'$ satisfying $d_\Gamma(z',y)=d_\Gamma(z,y)$. Since $z$ maximizes  $d_\Gamma(z,y) + d_\Gamma(z,y')$,
any neighbor of $z$ which lies farther away from $y$ must lie closer to $y'$. But there is exactly one such 
vertex. In all, we have $| \Gamma(z) | = 1 + p_{11}^1 + 1$.  But $\Gamma$ is $K_{2,1,1}$-free so the
neighborhood of any vertex is partitioned into cliques of size $p_{11}^1+1$. We find that $p_{11}^1+1$ 
divides $p_{11}^1+2$. This can only happen if $p_{11}^1=0$; i.e., $\Gamma$ is triangle-free. But then 
$z$ has degree two and $\Gamma$ must be a polygon.   $\Box$
\end{proof}

Our final two results deal with the special case where graph $\Gamma$ has diameter two.

\begin{theorem}
\label{Tdiam2}
Let $(X,\cR)$ be a symmetric association scheme and let $\Gamma=(X,R_1)$  be the graph 
associated to a connected basis relation. If $\Gamma$ has diameter two and $|X| > v_1(t-1)+2$, 
then  $\Gamma$ has vertex connectivity at least $t+1$ unless $t=v_1$.  
\end{theorem}

\begin{proof}
Let $T$ be a minimal disconnecting set of size at most $t$. For each $y\in T$, we 
use the fact that any two vertices have at least one common neighbor to obtain
$$ \left| \bigcup_{y \neq y' \in T} \Gamma(y') \right|  \le (v_1-1)(t-1)+1 $$
so that there is some $x\in X \setminus T$ not adjacent to any element of $T$ except possibly $y$.
Let $B$ be the component of 
$\Gamma \setminus T$ containing $x$.  Since $\Gamma$ has diameter two, $x \sim y$ and 
every $z \in X \setminus (B\cup T)$ must also be adjacent to $y$.  Swapping roles of the vertices in 
$T$, we find that, for every $y$ in $T$, there is some vertex $x$ (necessarily in $B$) 
with $\Gamma(x) \cap T = \{y\}$. But this implies that every 
$z \in X \setminus (B\cup T)$ is adjacent to every vertex in $T$,  so $T = \Gamma(z)$
for every $z \not\in B \cup T$.  $\Box$
\end{proof}

\begin{remark}
We expect very few exceptions to arise here. If $t=v_1$, then we find that 
$X \setminus (B\cup T)=\{z\}$ is a singleton and all but at most $v_1-2$ elements of 
$B$ have exactly one neighbor in $T=\Gamma(z)$.  With $|X| \ge v_1^2 - v_1 + 3$  so
close to the Moore bound, does this condition force $\Gamma$ to be a Moore graph?
\end{remark}

\begin{theorem}
\label{Tcutsize3}
Let $(X,\cR)$ be a symmetric association scheme and let $\Gamma=(X,R_1)$  be the graph 
associated to a connected basis relation. If $\Gamma$ has diameter two, then either
$\Gamma$ has vertex connectivity at least four or
$\Gamma$ is isomorphic to one of the following graphs: the  4-cycle, the 5-cycle, $K_{3,3}$, 
the Petersen graph.
\end{theorem}

\begin{proof}
Let $T=\{y_1,y_2,y_3\}$ be a minimal disconnecting set of size three.

\noindent \underline{\sl Case (i):} $T \subseteq a^\bot$ for some $a\in X$.

 By Corollary \ref{C1}, we have $T=\Gamma(a)$ and $\Gamma$ has valency three; i.e., $\Gamma \cong K_{3,3}$.

\medskip

\noindent \underline{\sl Case (ii):} Assume $T$ is not contained in $a^\bot$ for any vertex $a$.

In view of Theorem \ref{Tdiam2}, we may assume $|X| \le 2v_1  + 2$. (There is no cubic graph on nine vertices.)
Let $B$ and $B'$ denote the vertex sets of two distinct connected components of $\Gamma \setminus T$
and assume, without loss of generality, that $|B| \le |B'|$. Then we have $|B| \le \frac{ |X| - 3}{2}$. So
$|B| - 1 \le v_1 - 2$. In view of Case (i), 
we may assume each $x\in B$ is adjacent to exactly two members of $T$ and every pair of distinct
vertices in $B$ is adjacent. This forces $|B|= v_1-1$. 
Looking at $x \sim x'$ in $B$, we find that $p_{11}^1 \ge |B|-2 + 1$
since $x$ and $x'$ must share a common neighbor in $T$.  Now compare this to some $y \in T$. Since 
we are not in Case (i), some $y\in T$ is not adjacent to any other element of $T$. For this $y$, choose  some 
neighbor $z$ of $y$ where $z\in B$ if $|\Gamma(y) \cap B| \le \frac{v_1}{2}$ and $z \in B'$ if  
$|\Gamma(y) \cap B| > \frac{v_1}{2}$. The number of common neighbors of $y$ and $z$ is then at
most $\frac{v_1}{2} - 1$. The inequalities $v_1 -2 \le p_{11}^1 \le \frac{v_1}{2} - 1$ then imply that 
$\Gamma$ is a polygon, which is impossible as $T$ was chosen to be minimal.  $\Box$
\end{proof}

\section*{Acknowledgments}

We are grateful to Sebastian Cioab\u{a}, Gavin King and Jason Williford for helpful conversations on these topics.
We thank the referee for his or her careful reading of the manuscript. This work grew out of investigations supported by the National Security Agency.

\end{document}